\documentclass[10pt,twoside]{article}
\usepackage{mathrsfs}
\usepackage{amssymb}
\usepackage{amsmath,color}
\usepackage[all]{xy}
\usepackage{amsthm}
\usepackage{url}
\usepackage{setspace}
\usepackage{hyperref}
\hypersetup{colorlinks,linkcolor=blue,citecolor=blue}     %%% 为引用插入超链接

\usepackage{tikz}
\usepackage{stmaryrd}
\usepackage{xcolor}
\usetikzlibrary{arrows,calc}
\usepackage{etex}

\numberwithin{equation}{section}

 \def\Hom{\mbox{\rm Hom}}

\def\Im{\mbox{\rm Im}}

\def\P{\mathcal {P}}
\def\A{\mathcal{A}} 
\def\Id{\mbox{\rm Id}\,} \def\Im{\mbox{\rm Im}\,}

\def\E{\mathbb{E}}
\def\s{\mathfrak{s}}

\def\B{\mathcal {B}}\def\C{\mathcal {C}}
\def\X{\mathcal{X}}
\def\Y{\mathcal{Y}}

\def\I{\mathcal{I}}

\newcommand{\bsm}{\begin{smallmatrix}}
\newcommand{\esm}{\end{smallmatrix}} %%%%%%\bsm

\usepackage{geometry}
\usepackage{paralist}

\newtheorem{theorem}{Theorem}[section]
\newtheorem{proposition}[theorem]{Proposition}
\newtheorem{definition}[theorem]{Definition}
\newtheorem{remark}[theorem]{Remark}
\newtheorem{lemma}[theorem]{Lemma}

\newtheorem{corollary}[theorem]{Corollary}
\newtheorem{condition}[theorem]{Condition}
\newtheorem{theorem*}{Theorem}

\newcommand{\Ker}{\operatorname{Ker}}

\newcommand{\RNum}[1]{\uppercase\expandafter{\romannumeral #1\relax}}% 罗马数字

\title{ \bf $n$-cotorsion pairs in a recollement of extriangulated categories
\footnotetext{$^\ast$Corresponding author.~~Xin Ma is supported by Central Plains Science and Technology Innovation Youth Top-notch Talent and Henan University of Engineering (DKJ2019010). Panyue Zhou is supported by the Scientific Research Fund of Hunan Provincial Education Department (Grant No. 24A0221).
} }
\author{Xin Ma and Panyue Zhou$^\ast$}
\date{ }
\begin{document}

\baselineskip=16pt
\maketitle

\begin{abstract}
\begin{spacing}{1.3}
Let $(\mathcal{A}, \mathcal{B}, \mathcal{C})$ be a recollement of extriangulated categories.
In this paper,
we first show how to obtain an $n$-cotorsion pair in $\mathcal{B}$ from given $n$-cotorsion pairs in $\mathcal{A}$ and $\mathcal{C}$.
Conversely, we prove that an $n$-cotorsion pair in $\mathcal{B}$ can induce $n$-cotorsion pairs in $\mathcal{A}$ and $\mathcal{C}$  under suitable conditions.
As applications, several related results %\checks{and examples }
are provided to illustrate our construction.
\end{spacing}
\vspace{2mm}

\hspace{-5mm}\textbf{Keywords:} extriangulated category; recollement; $n$-cotorsion pair; $(n+1)$-cluster tilting subcategory\\ [0.2cm]
\end{abstract}

\pagestyle{myheadings}
\markboth{\rightline {\scriptsize X. Ma, P. Zhou\hspace{2mm}}}%%%%%%%%%%%    %%%%%%%%    %%%%%%%%%    %%%%%%% 页上角
{\leftline{\scriptsize $n$-cotorsion pairs in a recollement of extriangulated categories}}%%%%%%%   %%%%%    %%%%%%    %%%%%%%页上角

%\section{Introduction} %delete * to number this section

%\tableofcontents

\section{Introduction}

The notion of an extriangulated category was introduced by Nakaoka and Palu \cite{Na} as a broad unifying framework that simultaneously generalizes exact categories and triangulated categories.
This framework is sufficiently flexible to accommodate numerous examples which are neither exact nor triangulated in nature (see \cite{HZZ20P, INY18A, Na, ZZ18T} for detailed constructions).
In this setting, a variety of classical results from exact and triangulated categories can be reformulated and understood in a common language, thereby allowing techniques and ideas from both theories to be applied in a unified manner \cite{GNP21, HZZ21G, HZZ20P, HZZ21P, HZ21R, INY18A, Liu, ZZ20T, ZZ21G}.
Such a unification not only facilitates the transfer of methods but also provides new perspectives for extracting structural properties that are shared between these two important classes of categories.

Building upon this foundation, Wang, Wei, and Zhang \cite{WWZ20R} formulated the concept of recollements for extriangulated categories.
This concept extends the classical notion of recollements in abelian categories and in triangulated categories, first established by Beilinson, Bernstein, and Deligne \cite{BBD}, and reveals deep connections between these two settings.
Indeed, the interplay between abelian and triangulated recollements has been well-documented.
In particular, Chen \cite{C13C} studied cotorsion pairs in a recollement of triangulated categories, and further established a method for constructing recollements of abelian categories from a given recollement of triangulated categories.
In recent years, considerable attention has been devoted to gluing techniques in the context of recollements of extriangulated categories.
For example, He, Hu, and Zhou \cite{HHZ21T} studied the gluing of torsion pairs;
Ma, Zhao, and Zhuang \cite{MZZ} examined the gluing of resolving subcategories together with their associated resolution dimensions;
and Gu, Ma, and Tan \cite{GMT} analyzed the influence of recollements on homological invariants such as global dimension and extension dimension.
More recently, Ma and Zhou \cite{MZ25D} investigated the relationship between coresolution dimensions of subcategories and the construction of hereditary cotorsion pairs in a recollement of extriangulated categories.
It should be emphasized that many other related contributions exist in the literature, which are not enumerated here, collectively underscoring both the significance and the active development of research in this area.

Cao, Wei, and Wu \cite{CWW24R} investigated the interplay between $n$-cotorsion pairs in a recollement of abelian categories. Their results were subsequently extended to the setting of extriangulated categories by He and He \cite{HH25n}, thereby broadening the applicability of the theory to a more general categorical framework.
The concept of $n$-cotorsion pairs in extriangulated categories was formulated by Chang, Liu, and Zhou \cite{CLZ25M} as a simultaneous generalization of $n$-cotorsion pairs in triangulated categories \cite{CZ25M} and cotorsion pairs in extriangulated categories \cite{Na}. It is worth emphasizing that this notion is distinct from the version studied in \cite{HZ}; in particular, the $n$-cotorsion pairs considered in \cite{CLZ25M} are stronger than the weaker form introduced by He and Zhou \cite{HZ}. Such refinements play an important role in unifying and extending homological constructions across different categorical contexts.

In the present work, we focus on constructing $n$-cotorsion pairs within a recollement of extriangulated categories, as defined in Definition~\ref{def-n-cotor}. In Section~2, we recall fundamental notions and properties concerning extriangulated categories and recollements, which will be used throughout the paper. In Section~3, we consider a recollement $(\mathcal{A},\mathcal{B},\mathcal{C})$ of extriangulated categories and establish a method to construct an $n$-cotorsion pair in $\mathcal{B}$ from given $n$-cotorsion pairs in $\mathcal{A}$ and $\mathcal{C}$ (Theorem~\ref{main-thm1}). As an application, we obtain a construction of $(n+1)$-cluster tilting subcategories in the recollement $(\mathcal{A},\mathcal{B},\mathcal{C})$ (Corollary~\ref{main-1-2}). In particular, applying Corollary~\ref{main-1-2} to the triangulated case recovers the result of \cite[Theorem~3.4]{LZZ24R}. Moreover, we show that an $n$-cotorsion pair in $\mathcal{B}$ can, under suitable conditions, give rise to $n$-cotorsion pairs in $\mathcal{A}$ and $\mathcal{C}$ (Theorem~\ref{main-thm2}). Several corollaries (Corollaries~\ref{main-2-1}--\ref{main-2-5}) illustrate these results, further demonstrating the scope and flexibility of the approach.

Throughout this paper, all subcategories are assumed to be full, additive, and closed under isomorphisms.

\section{Preliminaries}

We begin by recalling several key notions and basic properties of extriangulated categories, following the formulation in \cite{Na}. For detailed definitions and illustrative examples, the reader is referred to \cite{Na}. Here, we restrict ourselves to a concise summary that will be used throughout the paper.

Let $\mathcal{C}$ be an additive category, and let
\[
\mathbb{E} \colon \mathcal{C}^{\mathrm{op}} \times \mathcal{C} \longrightarrow \mathrm{Ab}
\]
be an additive bifunctor, where $\mathrm{Ab}$ denotes the category of abelian groups. For any pair of objects $A, C \in \mathcal{C}$, an element $\delta \in \mathbb{E}(C, A)$ is called an \emph{$\mathbb{E}$-extension} from $A$ to $C$.

A \emph{realization} of $\mathbb{E}$ is a rule $\mathfrak{s}$ which assigns to each $\mathbb{E}$-extension $\delta \in \mathbb{E}(C, A)$ an equivalence class of sequences of the form
\[
\mathfrak{s}(\delta) \;=\; \xymatrix@C=0.8cm{[A \ar[r]^{x} & B \ar[r]^{y} & C]},
\]
in such a way that the diagrams required in \cite[Definition~2.9]{Na} commute. Intuitively, $\mathfrak{s}$ serves as a bridge connecting abstract extension classes with concrete diagrammatic realizations inside $\mathcal{C}$.

An \emph{extriangulated category} is then a triple $(\mathcal{C}, \mathbb{E}, \mathfrak{s})$ satisfying the following conditions:
\begin{itemize}
    \item[(1)] $\mathbb{E} \colon \mathcal{C}^{\mathrm{op}} \times \mathcal{C} \to \mathrm{Ab}$ is an additive bifunctor.
    \item[(2)] $\mathfrak{s}$ is an additive realization of $\mathbb{E}$.
    \item[(3)] The pair $(\mathbb{E}, \mathfrak{s})$ satisfies the compatibility conditions specified in \cite[Definition~2.12]{Na}.
\end{itemize}

When there is no ambiguity, we simply write $\mathcal{C}:= (\mathcal{C}, \mathbb{E}, \mathfrak{s})$ to denote an extriangulated category. This framework, as established in \cite{Na}, simultaneously encompasses exact categories and triangulated categories, while also admitting examples beyond both settings, thus providing a versatile environment for homological and categorical constructions.

Let $\X, \Y$ be subcategories of $\mathcal{C}$.
We recall the following notations from \cite{Na}.
\begin{itemize}
    \item[(1)]
    Set
    \begin{align*}
{\rm Cone}(\X,\Y):=\{C\in \C\mid \text{ there is an }\E\text{-triangle}\xymatrix@C=15pt{X\ar[r] & Y \ar[r] & C\ar@{-->}[r]&}\text{ with }X\in \X\text{ and }Y\in \Y \};\\
{\rm Cocone}(\X,\Y):=\{C\in \C\mid \text{ there is an }\E\text{-triangle}\xymatrix@C=15pt{C\ar[r] & X \ar[r] & Y\ar@{-->}[r]&}\text{ with }X\in \X\text{ and }Y\in \Y \}.
\end{align*}
    \item[(2)] $\mathcal{X}$ is said to be closed under $\E$-extensions
  if for any $\E$-triangle
      $A\stackrel{}{\longrightarrow}B\stackrel{}{\longrightarrow}C\stackrel{}\dashrightarrow$
in $\C$, $A, C \in \mathcal{X}$, then $B \in \mathcal{X}$ as well.
\end{itemize}

\begin{definition}{\rm(\cite[Definitions 3.23 and 3.25]{Na})
Let $\mathcal{C}$ be an extriangulated category.
\begin{itemize}
\item[(1)] An object $P$ in $\mathcal{C}$ is called {\em projective} if for any $\mathbb{E}$-triangle $A\stackrel{x}{\longrightarrow}B\stackrel{y}{\longrightarrow}C\stackrel{}\dashrightarrow$ and any morphism $c$ in $\mathcal{C}(P,C)$, there exists $b$ in $\mathcal{C}(P,B)$ such that $yb=c$.
We denote the subcategory of projective objects in $\mathcal{C}$ by $\mathcal{P}(\mathcal{C})$.
Dually, the {\em injective} objects are defined, and the subcategory of injective objects in $\mathcal{C}$ is denoted by $\mathcal{I}(\mathcal{C})$.
\item[(2)] We say that $\mathcal{C}$ {\em has enough projectives} if for any object $M\in\mathcal{C}$, there exists an $\mathbb{E}$-triangle $$A\stackrel{}{\longrightarrow}P\stackrel{}{\longrightarrow}M\stackrel{}\dashrightarrow$$ satisfying $P\in\mathcal{P}(\mathcal{C})$. Dually, we can define $\mathcal{C}$ {\em having enough injectives}.
\end{itemize}}
\end{definition}

In what follows, all extriangulated categories are assumed to have enough projective and enough injective objects, and to satisfy the following condition.

\begin{condition}\label{WIC} {\rm \bf (WIC)} {\rm (see \cite[Condition 5.8]{Na})} Let $f:X\rightarrow Y$ and  $g:Y\rightarrow Z$ be any composable pair of morphisms in $\mathcal{C}$.
\begin{itemize}
\item[\rm (1)]  If $gf$ is an inflation, then $f$ is an inflation.
\item[\rm (2)] If $gf$ is a deflation, then $g$ is a deflation.
\end{itemize}
\end{condition}

\begin{remark}
Note that every extriangulated category
$\C$ under consideration is assumed to satisfy the above condition, which is equivalent to
$\C$ being weakly idempotent complete (see {\rm \cite[Proposition~2.7]{K}}).
\end{remark}

For a subcategory $\X\subseteq\C$, put $\Omega^0\X=\X$, and define $\Omega^k\X$ for $k>0$ inductively by
$$\Omega^k\X=\Omega(\Omega^{k-1}\X)={\rm Cocone}(\P,\Omega^{k-1}\X).$$
We call $\Omega^k\X$ the $k$-th syzygy of $\X$.

Dually we define the $k$-th cosyzygy $\Sigma^k\X$ by
$\Sigma^0\X=\X$ and $\Sigma^k\X={\rm Cone}(\Sigma^{k-1}\X,\I)$ for $k>0$.

Liu and Nakaoka \cite{Liu} defined higher extension groups in an extriangulated category as $$\E(A,\Sigma^kB)\simeq\E(\Omega^kA,B)~~\mbox{for}~~k\geq0.$$
For convenience, we denote $\E(A,\Sigma^kB)\simeq\E(\Omega^kA,B)$ by $\E^{k+1}(A,B)$ for $k\geq0$.
They proved the following result, which is an important tool in relative homological theory of extriangulated categories and can be called ``dimension shifting" as a strategy.
\begin{lemma}\label{lem-long}
Let $\C$ be an extriangulated category. Assume that
$$\xymatrix{A\ar[r]^{f}& B\ar[r]^{g}&C\ar@{-->}[r]^{\delta}&}$$
is an $\E$-triangle in $\C$. Then for any object $X\in\C$ and $k\geq1$, we have the following exact sequences:
$$\cdots\xrightarrow{}\E^k(X,A)\xrightarrow{}\E^k(X,B)\xrightarrow{}\E^k(X,C)\xrightarrow{}\E^{k+1}(X,A)\xrightarrow{}\E^{k+1}(X,B)\xrightarrow{}\cdots;$$
$$\cdots\xrightarrow{}\E^k(C,X)\xrightarrow{}\E^k(B,X)\xrightarrow{}\E^k(A,X)\xrightarrow{}\E^{k+1}(C,X)\xrightarrow{}\E^{k+1}(B,X)\xrightarrow{}\cdots.$$
\end{lemma}

At the end of this section, we recall the concept of recollements of extriangulated categories \cite{WWZ20R}, which gives a simultaneous generalization of recollements of triangulated categories and abelian categories (see \cite{BBD, Fr}).
Some details can be found in \cite{WWZ20R}, we omit them here.

\begin{definition}\label{def-rec}{\rm(\cite[Definition 3.1]{WWZ20R})
Let $\mathcal{A}$, $\mathcal{B}$ and $\mathcal{C}$ be three extriangulated categories. A \emph{recollement} of $\mathcal{B}$ relative to
$\mathcal{A}$ and $\mathcal{C}$, denoted by ($\mathcal{A}$, $\mathcal{B}$, $\mathcal{C}$), is a diagram
\begin{equation*}\label{recolle}
  \xymatrix{\mathcal{A}\ar[rr]|{i_{*}}&&\ar@/_1pc/[ll]|{i^{*}}\ar@/^1pc/[ll]|{i^{!}}\mathcal{B}
\ar[rr]|{j^{\ast}}&&\ar@/_1pc/[ll]|{j_{!}}\ar@/^1pc/[ll]|{j_{\ast}}\mathcal{C}}
\end{equation*}
given by two exact functors $i_{*},j^{\ast}$, two right exact functors $i^{\ast}$, $j_!$ and two left exact functors $i^{!}$, $j_\ast$, which satisfies the following conditions:
\begin{itemize}
  \item [\rm (R1)] $(i^{*}, i_{\ast}, i^{!})$ and $(j_!, j^\ast, j_\ast)$ are adjoint triples.
  \item [\rm (R2)] $i_\ast$, $j_!$ and $j_\ast$ are fully faithful.
  \item [\rm (R3)] $\Im i_{\ast}=\Ker j^{\ast}$.
  \item [\rm (R4)] For each $B\in\mathcal{B}$, there exists a left exact $\mathbb{E}$-triangle sequence
$$
  \xymatrix{i_\ast i^!B\ar[r]^-{\theta_B}&B\ar[r]^-{\vartheta_B}&j_\ast j^\ast B\ar[r]&i_\ast A}
$$
 in $\mathcal{B}$ with $A\in \mathcal{A}$, where $\theta_B$ and  $\vartheta_B$ are given by the adjunction morphisms.
  \item [\rm (R5)] For each $B\in\mathcal{B}$, there exists a right exact $\mathbb{E}$-triangle sequence
$$
  \xymatrix{i_\ast\ar[r]A' &j_! j^\ast B\ar[r]^-{\upsilon_B}&B\ar[r]^-{\nu_B}&i_\ast i^\ast B&}
$$
in $\mathcal{B}$ with $A'\in \mathcal{A}$, where $\upsilon_B$ and $\nu_B$ are given by the adjunction morphisms.
\end{itemize}}
\end{definition}

 In particular, if $\A$, $\B$ and $\C$ are abelian categories, then $(\A,\B,\C)$ is a recollement if it satisfies (R1), (R2) and $j^{*}i_{*}=0$ in sense of \cite{Fr}.

\vspace{2mm}

We collect some properties of recollements of extriangulated categories (see \cite{WWZ20R}).

\begin{lemma}\label{lem-rec} Let ($\mathcal{A}$, $\mathcal{B}$, $\mathcal{C}$) be a recollement of extriangulated categories.

$(1)$ All the natural transformations
$$i^{\ast}i_{\ast}\Rightarrow\Id_{\A},~\Id_{\A}\Rightarrow i^{!}i_{\ast},~\Id_{\C}\Rightarrow j^{\ast}j_{!},~j^{\ast}j_{\ast}\Rightarrow\Id_{\C}$$
are natural isomorphisms.
Moreover, $i^{!}$, $i^{*}$ and $j^{*}$ are dense.

$(2)$ $i^{\ast}j_!=0$ and $i^{!}j_\ast=0$.

$(3)$ $i^{\ast}$ preserves projective objects and $i^{!}$ preserves injective objects.

$(3')$ $j_{!}$ preserves projective objects and $j_{\ast}$ preserves injective objects.

$(4)$ If $i^{!}$ (resp., $j_{\ast}$) is  exact, then $i_{\ast}$ (resp., $j^{\ast}$) preserves projective objects.

$(4')$ If $i^{\ast}$ (resp., $j_{!}$) is  exact, then $i_{\ast}$ (resp., $j^{\ast}$) preserves injective objects.

$(5)$ If $i^{!}$ is exact, then $j_{\ast}$ is exact.

$(5')$
If $i^{\ast}$ is exact, then $j_{!}$ is  exact.

$(6)$ If $i^{!}$ is exact, for each $B\in\mathcal{B}$, there is an $\mathbb{E}$-triangle
  \begin{equation*}\label{third}
  \xymatrix{i_\ast i^! B\ar[r]^-{\theta_B}&B\ar[r]^-{\vartheta_B}&j_\ast j^\ast B\ar@{-->}[r]&}
   \end{equation*}
 in $\mathcal{B}$ where $\theta_B$ and  $\vartheta_B$ are given by the adjunction morphisms.

$(6')$ If $i^{\ast}$ is exact, for each $B\in\mathcal{B}$, there is an $\mathbb{E}$-triangle
  \begin{equation*}\label{four}
  \xymatrix{ j_! j^\ast B\ar[r]^-{\upsilon_B}&B\ar[r]^-{\nu_B}&i_\ast i^\ast B\ar@{-->}[r]&}
   \end{equation*}
in $\mathcal{B}$ where $\upsilon_B$ and $\nu_B$ are given by the adjunction morphisms.
\end{lemma}

\section{Our main results}

Let $\X$ be a subcategory of $\C$,
We define the \emph{right $k$-th orthogonal} of $\mathcal{X}$ as
\[
\mathcal{X}^{\perp_k} := \{ N \in \mathcal{C} \mid \mathbb{E}^k_{\C}(X, N) = 0 \text{ for any } X\in \X\}.
\]
Dually, the \emph{left $k$-th orthogonal} of $\mathcal{X}$ is defined by
\[
{}^{\perp_k} \mathcal{X} := \{ M \in \mathcal{C} \mid \mathbb{E}^k_{\C}(M, X) = 0 \text{ for any } X\in \X \}.
\]
Now we recall the notion of cotorsion pairs.

\begin{definition}{\rm (\cite[Definition 4.1]{Na})}\label{DefCotors}
{\rm Let $\X,\Y\subseteq\C$ be a pair of subcategories which are closed
under direct summands. The pair $(\X,\Y)$ is called a {\it cotorsion pair} in $\C$ if it satisfies the following conditions.
\begin{enumerate}
\item[(1)] $\E(\X,\Y)=0$.
\item[(2)] For any $C\in\C$, there exist two $\E$-triangles
$$\xymatrix@C=15pt{Y^1\ar[r]& X^1\ar[r]& C\ar@{-->}[r]&}$$
$$\xymatrix@C=15pt{C\ar[r]& Y_1\ar[r]& X_1\ar@{-->}[r]&}$$
satisfying $X^1, X_1 \in\X,Y^1, Y_1 \in\Y$.
\end{enumerate}
}
\end{definition}
If $(\X,\X)$ is a cotorsion pair in $\C$, the subcategory $\X$ is called the cluster tilting subcategory of $\C$ in the sense of \cite[Remark 2.11]{ZZ19C}.

We recall the following notions from \cite{AR}, where the concepts of approximations and finiteness of subcategories were introduced to describe how objects in a category can be effectively approximated by objects in a given subcategory.

\begin{definition}
Let $\mathcal{C}$ be an extriangulated category and let $\mathcal{X}$ be a subcategory of $\mathcal{C}$.
\begin{itemize}
  \item[\rm (1)] For an object $C \in \mathcal{C}$, a \emph{right $\mathcal{X}$-approximation} of $C$ is a morphism $f \colon X \to C$ with $X \in \mathcal{X}$ such that, for every $Y \in \mathcal{X}$, the induced sequence
 $$
    \Hom_{\mathcal{C}}(Y,X) \xrightarrow{{\rm Hom}_{\mathcal C}(Y,f)} \Hom_{\mathcal{C}}(Y,C) \to 0
  $$
  is exact. Dually, a \emph{left $\mathcal{X}$-approximation} of $C$ is defined.
  \item[\rm (2)] The subcategory $\mathcal{X}$ is called \emph{contravariantly finite} in $\mathcal{C}$ if every object $C \in \mathcal{C}$ admits a right $\mathcal{X}$-approximation, and \emph{covariantly finite} in $\mathcal{C}$ if every object $C \in \mathcal{C}$ admits a left $\mathcal{X}$-approximation.
\end{itemize}
\end{definition}

The following result gave a characterization of cotorsion pairs.
\begin{proposition}{\rm (\cite[Proposition 2.5]{CLZ25M})}\label{prop-1}
Let $\C$ be an extriangulated category, and let $\X$ and $\Y$ be subcategories of $\C$ which are closed under direct summands.
Then $(\X, \Y)$ is a cotorsion pair if and only if the following conditions are satisfied.
\begin{itemize}
  \item [\rm (1)]~ $\X={^{\perp_1}\Y}$.
  \item [\rm (2)]~ $\Y=\X^{\perp_1}$.
  \item [\rm (3)]~ $\X$ is contravariantly finite and $\Y$ is covariantly finite in $\C$.
\end{itemize}
\end{proposition}

\begin{definition}{\rm (\cite[Definition 5.3]{Liu})}
Let $\mathcal{C}$ be an extriangulated category. A subcategory $\mathcal{X} \subseteq \mathcal{C}$ is called \emph{$(n+1)$-cluster tilting} if the following conditions hold:
\begin{enumerate}
\item[{\rm (1)}] $\mathcal{X}$ is both contravariantly and covariantly finite in $\mathcal{C}$;
\item[{\rm (2)}] $M \in \mathcal{X}$ if and only if $\mathbb{E}^k(\mathcal{X}, M) = 0$ for all $1 \leq k \leq n$;
\item[{\rm (3)}] $M \in \mathcal{X}$ if and only if $\mathbb{E}^k(M, \mathcal{X}) = 0$ for all $1 \leq k \leq n$.
\end{enumerate}
\end{definition}

Now we recall the notion of $n$-cotorsion pairs, which is a generalization of the cotorsion pairs in the sense of Nakaoka and Palu \cite{Na}.

\begin{definition}
{\rm (\cite[Definition 3.1]{CLZ25M})}
\label{def-n-cotor}
Let $\C$ be an extriangulated category and $\X,\Y$ be two subcategories of $\C$ which are closed
under direct summands. The pair $(\X,\Y)$ is called an $n$-\emph{cotorsion pair} if the following conditions hold.
\begin{enumerate}
  \item [\rm (1)] $\X=\bigcap\limits_{k=1}^{n}{^{\perp_k}}\Y$; %~~ for all $1\leq k\leq n$
  \item [\rm (2)] $\Y=\bigcap\limits_{k=1}^{n}\X^{\perp_k}$;%~~ for all $1\leq k\leq n$
   \item [\rm (3)] $\X$ is contravariantly finite and $\Y$ is covariantly finite.
\end{enumerate}
\end{definition}

\begin{remark}\label{remark-nc}
Let $(\C,\E,\s)$ be an extriangulated category with enough projectives and enough injectives.
\begin{itemize}
\item[\rm (1)] It is clear that $\mathcal{X}$ and $\mathcal{Y}$ are closed under $\E$-extensions, and $\mathcal{P(C)}\subseteq \X$, $\mathcal{I(C)}\subseteq\Y$.

    \item[\rm (2)] It is easy to verify that
    $\mathcal{X}$ is $(n+1)$-cluster tilting subcategory if and only if $(\mathcal{X},\mathcal{X})$ is an $n$-cotorsion pair.

    \item[\rm (3)]
By Proposition \ref{prop-1}, the concept of $1$-cotorsion pair is compatible with the classical definition of a cotorsion pair in the sense of Nakaoka and Palu \cite{Na}.

\item[\rm (4)] The concept of $n$-cotorsion pair defined in this paper is a  weaker version than that in \cite{HZ} by He and Zhou,
see \cite[Remark 3.4 and Example 3.5]{CZ25M}.

\item[\rm (5)]
When $\C$ is an abelian category, an $n$-cotorsion pair defined in \cite{HMP} can imply an $n$-cotorsion pair in the sense of Definition \ref{def-n-cotor}. When $\C$ is a triangulated category, an $n$-cotorsion pair in the sense of Definition \ref{def-n-cotor} is compatible with Chang and Zhou defined in \cite{CZ25M}.
\end{itemize}
\end{remark}

\subsection{From $\A$ and $\B$ to $\C$}

The following results are easy and useful in the sequel.

\begin{lemma}{\rm (\cite[Lemma 4.3]{MZ25D})}\label{E-xt}
Let $\A$ and $\B$ be extriangulated categories, and let $F : \A \to \B$ be a functor admitting a right adjoint functor $G$.
For any $X\in \A$, $Y\in \B$ and any $i\geq 1$, if one of the following conditions is satisfied
 \begin{itemize}
 \item[\rm (1)] If $F$ is an exact functor and preserves projective objects;
 \item[\rm (2)] If $G$ is an exact functor and preserves injective objects;
 \end{itemize}
then we have
$$\mathbb{E}_{\B}^{i}(F(X),Y)\cong \mathbb{E}_{\A}^{i}(X,G(Y)).$$
\end{lemma}

Combing Lemmas \ref{lem-rec} and \ref{E-xt}, we have the following result immediately.

\begin{proposition}{\rm (cf. \cite[Lemma 2.5]{HH25n})}\label{E-xt-1}
    Let $(\A, \B, \C)$ be a recollement of extriangulated categories. If $i^{*}$ and $i^!$ are exact, then for any $A\in \A$, $B\in \B$, $C\in \C$ and any integer $i\geq 1$, we have
   \begin{align*}
       &
   \E_{\A}^{i}(i^{*}B,A)\cong \E_{\B}^{i}(B,i_{*}A),\ \E_{\B}^{i}(i_{*}A,B)\cong \E_{\A}^{i}(A,i^{!}B),\\
  &
  \E_{\B}^{i}(j_{!}C,B)\cong \E_{\C}^{i}(C,j^{*}B),
   \ \E_{\C}^{i}(j^{*}B,C)\cong \E_{\B}^{i}(B,j_{*}C).
   \end{align*}
\end{proposition}

As a similar argument to that of \cite[Lemmas 3.1 and 3.2]{LW10F} and \cite[Lemmas 3.5 and 3.6]{MXZ22S}, we have the following result, its proof is omitted here.

\begin{lemma}\label{lem-con-cov}
Let $\xymatrix@C=15pt{j^{*}:\mathcal{B}\ar[r]&\mathcal{C}}$ be an additive functor between extriangulated categories $\mathcal{B}$ and $\mathcal{C}$.
Then we have the following statements.
\begin{itemize}
    \item[\rm (1)] If $j^{*}$ has a left adjoint $j_{!}$ and $\mathcal{Y}$ is a covariantly finite subcategory of $\mathcal{B}$, then $j^{*}\mathcal{X}$ is covariantly finite in $\mathcal{C}$.
    \item[\rm (2)] If $j^{*}$ has a right adjoint $j_{*}$ and $\mathcal{X}$ is a contravariantly finite subcategory of $\mathcal{B}$, then $j^{*}\mathcal{X}$ is contravariantly finite in $\mathcal{C}$.
\end{itemize}
\end{lemma}

Before presenting the main result, we recall that in a recollement $(\mathcal{A},\mathcal{B},\mathcal{C})$ of extriangulated categories, the interplay between substructures in $\mathcal{A}$ and $\mathcal{C}$ often determines corresponding structures in the middle category $\mathcal{B}$. Motivated by this observation, we aim to construct an $n$-cotorsion pair in $\mathcal{B}$ from given $n$-cotorsion pairs in $\mathcal{A}$ and $\mathcal{C}$.
The following theorem establishes such a construction and shows its compatibility with the original $n$-cotorsion pairs in the outer terms of the recollement.

\begin{theorem}\label{main-thm1}
    Let $(\A, \B, \C)$ be a recollement of extriangulated categories, and let $(\mathcal{X'},\mathcal{Y'})$ and $(\mathcal{X''},\mathcal{Y''})$ be $n$-cotorsion pairs in $\A$ and $\C$ respectively. Set
    \begin{align*}
        \mathcal{X}=&\{B\in\B\mid j^{*}B\in\mathcal{X''},\ i^{*}B\in \mathcal{X'}\},\\
       \mathcal{Y}=&\{B\in\B\mid j^{*}B\in\mathcal{Y''},\ i^{!}B\in \mathcal{Y'}\}.
    \end{align*}
    If $i^{*}$ and $i^!$ are exact, then
    \begin{itemize}
    \item[\rm (1)]
    $i_{*}\mathcal{X'}\subseteq\mathcal{X}$, $i_{*}\mathcal{Y'}\subseteq\mathcal{Y}$, $j_{!}\mathcal{X''}\subseteq\mathcal{X}$, $j_{*}\mathcal{Y''}\subseteq\mathcal{Y}$.
        \item[\rm (2)] $(\mathcal{X},\mathcal{Y})$ is an $n$-cotorsion pair in $\B$.
        \item[\rm (3)] $(\mathcal{X'},\mathcal{Y'})=(i^{*}\mathcal{X},i^{!}\mathcal{Y})$ and $(\mathcal{X''},\mathcal{Y''})=(j^{*}\mathcal{X},j^{*}\mathcal{Y})$.
    \end{itemize}
\end{theorem}
\begin{proof}
    (1)
 Since $i^{*}i_{*}\Rightarrow \Id_{\A}$ by Lemma \ref{lem-rec},
$i^{*}i_{*}\mathcal{X'}\cong \mathcal{X'}$.
Notice that $j^{*}i_{*}\mathcal{X'}=0$, so $i_{*}\mathcal{X'}\subseteq\mathcal{X}$.
  Since $j^{*}j_{*}\Rightarrow\Id_{\C}$ by Lemma \ref{lem-rec},
$j^{*}j_{*}\mathcal{Y''}\cong \mathcal{Y''}$. Notice that $i^{!}j_{*}\mathcal{Y''}=0$ by Lemma \ref{lem-rec}, then $j_{*}\mathcal{Y''}\subseteq\mathcal{Y}$.
Similarly, we can get the assertions that $i_{*}\mathcal{Y'}\subseteq\mathcal{Y}$ and $j_{!}\mathcal{X''}\subseteq\mathcal{X}$.

(2)
    Since $(\X',\Y')$ and $(\X'',\Y'')$ are $n$-cotosion pairs in $\A$ and $\C$ respectively by assumption,
    $\X'$ and $\X''$ are contravariantly finite subcategories in $\A$ and $\C$ respectively, and $\Y'$ and $\Y''$ are covariantly finite
    subcategories in $\A$ and $\C$ respectively.
By Lemma \ref{lem-con-cov}, we know that $i_*\X'$ and $j_!\X''$ are contravariantly finite subcategories in $\B$.

Let $B\in \B$.
Then there exists a right $j_{!}\X''$-approximation $f: X_{1}\rightarrow B$ of $B$
with $X_1\in j_{!}\X''$.
Notice that $\B$ has enough injective objects, so there exists an $\E$-triangle
$$\xymatrix{X_{1}\ar[r]^{\phi}&I\ar[r]^{\psi}&K_{1}\ar@{-->}[r]&}$$
in $\B$ with $I\in \mathcal{I(B)}$.
By \cite[Proposition 1.20]{Liu}, we have the following commutative diagram of $\mathbb{E}_\mathcal{B}$-triangles
\begin{align*}
\xymatrix{X_{1}\ar[d]^{f}\ar[r]^{\phi}&I\ar[d]^{l}\ar[r]^{\psi}&K_{1}\ar@{-->}[r]\ar@{=}[d]&\\
B\ar[r]^{k} & K \ar[r]& K_1\ar@{-->}[r]^{} &}
\end{align*}
and an $\E_{\B}$-triangle $$\xymatrix{X_{1}\ar[r]^-{\phi\choose f}&I\oplus B\ar[r]^-{(-l\ k)}&K\ar@{-->}[r]&}.$$
Let $g: X_{2}\rightarrow K$ be a right $i_{*}\X'$-approximation of $K$ with $X_2\in i_{*}\X'$.
By the dual of \cite[Proposition 1.20]{Liu}, we get the following commutative diagram of $\E_{\B}$-triangles
\begin{align*}
\xymatrix{X_{1}\ar@{=}[d]\ar[r]&X\ar[d]^{u'\choose u}\ar[r]^{m}&X_{2}\ar@{-->}[r]\ar[d]^{g}&\\
X_1\ar[r]^-{\phi \choose f} & I\oplus B \ar[r]^-{(-l\ k)}& K\ar@{-->}[r]^{} &}
\end{align*}
and the following $\E$-triangle
\begin{align*}
    \xymatrix@C=30pt{X\ar[r]^-{\tiny\begin{pmatrix}
m\\ u'\\ u
\end{pmatrix}}&X_2\oplus I\oplus B\ar[r]^-{(g\ l\ -k)}&K\ar@{-->}[r]&}
\end{align*} 
in $\B$.
One can see that $\X$ is closed under $\E$-extensions.
Notice that $X_{1}, X_{2}\in \X$ by (1), so $X\in \X$.

Now we claim that $u: X\rightarrow B$ is a right $\X$-approximation of $B$.
Let $h: \tilde{X}\rightarrow B$ be any morphism with $\tilde{X}\in \X$.
Since $i^{*}$ is exact, there exists an $\E$-triangle
\begin{align*}
\xymatrix{j_{!}j^{*}\tilde{X}\ar[r]^-{\upsilon_{\tilde{X}}}&\tilde{X}\ar[r]^-{\nu_{\tilde{X}}}&i_{*}i^{*}\tilde{X}\ar@{-->}[r]&}
\end{align*}
in $\B$ by Lemma \ref{lem-rec}.
Notice that $j^*\tilde{X}\in \X''$ and $i^*\tilde{X}\in \X'$, so $j_!j^*\tilde{X}\in j_{!}\X''$ and  $i_{*}i^*\tilde{X}\in i_{*}\X'$.
Consider the morphism
$h\upsilon_{\tilde{X}} : j_!j^{*}\tilde{X}\rightarrow B$.
Since $f: X_{1}\rightarrow B$ is a right $j_{!}\X''$-approximation of $B$, there exists $p: j_{!}j^{*}\tilde{X}\rightarrow X_1$ such that $h\upsilon_{\tilde{X}}=fp$.
On the other hand, consider the morphism $\phi p: j_{!}j^{*}\tilde{X}\rightarrow I $, since $I\in \mathcal{I(B)}$, there exists a morphism $h': \tilde{X}\rightarrow I$ such that $\phi p=h'\upsilon_{\tilde{X}}$.
Then we obtain the following commutative diagram
$$\xymatrix{
j_{!}j^{*}\tilde{X}\ar[d]^{p}\ar[r]^-{\upsilon_{\tilde{X}}}&\tilde{X}\ar[r]^-{\nu_{\tilde{X}}}\ar[d]^{h'\choose h}&i_{*}i^{*}\tilde{X}\ar@{-->}[d]^{p''}\ar@{-->}[r]&\\
X_1\ar[r]^-{\phi \choose f}&I\oplus B\ar[r]^-{(-l\ k)}&K\ar@{-->}[r]&}.$$
Notice that $g: X_{2}\rightarrow K$ is a right $i_{*}\X'$-approximation of $K$ and $i_{*}i^*\tilde{X}\in i_{*}\X'$, there exists a morphism $g': i_{*}i^{*}\tilde{X}\rightarrow X_{2}$ such that $p''=gg'$.
It follows that
$$(g\ l\ -k){\tiny\begin{pmatrix}
g'{\nu_{\tilde{X}}} \\ h' \\ h
\end{pmatrix}}=gg'{\nu_{\tilde{X}}}+lh'-kh=p''{\nu_{\tilde{X}}}+lh'-kh=0,$$
where ${\tiny\begin{pmatrix}
g'{\nu_{\tilde{X}}} \\ h' \\ h
\end{pmatrix}}: \tilde{X}\rightarrow X_{2}\oplus I\oplus B$.
Then there exists $\tilde{h}: \tilde{X}\rightarrow X$ such that
$${\tiny\begin{pmatrix}
g'\nu_{\tilde{X}} \\ h' \\ h
\end{pmatrix}}={\tiny\begin{pmatrix}
m \\ u' \\ u
\end{pmatrix}}\tilde{h},$$
and so $h=u\tilde{h}$.
Thus $u: X\rightarrow B$ is a right $\X$-approximation of $B$, and so $\X$ is contravariantly finite in $\B$.

Dually, we can prove that $\Y$ is covariantly finite in $\B$.

Now we claim that
$\mathcal{X}=\bigcap\limits_{k=1}^{n}{^{\perp_k}\mathcal{Y}}$.
For any $X\in \mathcal{X}$, since $i^*$ is exact, there exists the following $\E$-triangle
$$\xymatrix{j_{!}j^{*}X\ar[r]&X\ar[r]&i_{*}i^{*}X\ar@{-->}[r]&}$$
in $\B$ by Lemma \ref{lem-rec}.
For any $Y\in \mathcal{Y}$ and $1\leq k\leq n$,
applying the functor $\Hom_{\B}(-,Y)$ to the above $\E$-triangle, we get the following exact sequence
$$\xymatrix{\cdots\ar[r]&\E_{\B}^{k}(i_{*}i^{*}X,Y)\ar[r]&\E^{k}_{\B}(X,Y)\ar[r]&\E^{k}_{\B}(j_{!}j^{*}X,Y)\ar[r]&\cdots}.$$
Since $i^{*}X\in\mathcal{X'}$, $j^{*}X\in \mathcal{X''}$, $j^{*}Y\in \mathcal{Y''}$ and $i^{!}Y\in\mathcal{Y'}$,
by Proposiion \ref{E-xt-1},
$$\E^{k}_{\B}(i_{*}i^{*}X,Y)\cong\E^{k}_{\A}(i^{*}X,i^{!}Y)=0$$
and
$$\E^{k}_{\B}(j_{!}j^{*}X,Y)\cong\E^{k}_{\C}(j^{*}X,j^{*}Y)=0.$$
It follows that $\E_{\B}^{i}(X,Y)=0$, then $\mathcal{X}\subseteq\bigcap\limits_{k=1}^{n} {^{\perp_{k}}\mathcal{Y}}$
On the other hand, assume that $X\in\bigcap\limits_{k=1}^{n} {^{\perp_{k}}\mathcal{Y}}$.
So $\E^{k}_{\B}(X,Y)=0$ for any $Y\in \mathcal{Y}$ and $1\leq k\leq n$.
Then for any $Y'\in \mathcal{Y'}$ and $Y''\in \mathcal{Y''}$,
we have $\E^{k}_{\mathcal{A}}(i^{*}X,Y')\cong \E^{k}_{\mathcal{B}}(X,i_{*}Y')=0$
and $\E^{k}_{\C}(j^{*}X,Y'')\cong \E^{k}_{\mathcal{B}}(X,j_{*}Y'')=0$
by Proposition \ref{E-xt-1} and (1).
Then $i^{*}X\in\bigcap\limits_{k=1}^{n}{^{\perp_{k}}\mathcal{Y'}}=\mathcal{X'}$ and $j^{*}X\in \bigcap\limits_{k=1}^{n}{^{\perp_{k}}\mathcal{Y''}}=\mathcal{X''}$, and so $X\in \mathcal{X}$.
It follows that $\bigcap\limits_{k=1}^{n} {^{\perp_{k}}\mathcal{Y}}\subseteq\mathcal{X}$.
Thus $\mathcal{X}=\bigcap\limits_{k=1}^{n} {^{\perp_{k}}\mathcal{Y}}$.

Similarly, we can prove that
$\mathcal{Y}=\bigcap\limits_{k=1}^{n}{\mathcal{X}^{\perp_{k}}}$.
Thus $(\mathcal{X},\mathcal{Y})$ is an $n$-cotorsion pair in $\B$.

    (3)
  It is clear that $i^{*}\mathcal{X}\subseteq\mathcal{X'}$.
    Since $i^{*}i_{*}\Rightarrow \Id_{\A}$ by Lemma \ref{lem-rec} and $i_{*}\mathcal{X'}\subseteq\mathcal{X}$ by (1),
    $\mathcal{X'}\cong i^{*}i_{*}\mathcal{X'}\subseteq i^{*}\mathcal{X}$.
Then $i^{*}\mathcal{X}=\mathcal{X'}$.
Similarly, we can prove the assertions that
$i^{!}\mathcal{Y}=\mathcal{Y'}$, $j^{*}\mathcal{X}=\mathcal{X''}$ and $j^{*}\mathcal{Y}= \mathcal{Y''}$.
\end{proof}

Applying Theorem~\ref{main-thm1} to triangulated categories yields the following result, which generalizes \cite[Theorem~3.1]{C13C}. In particular, when $n=1$, Corollary~\ref{main-1-1} recovers Theorem~3.1 of Chen~\cite{C13C}.

\begin{corollary}
\label{main-1-1}
Let $(\A, \B, \C)$ be a recollement of triangulated categories, and let $(\mathcal{X'},\mathcal{Y'})$ and $(\mathcal{X''},\mathcal{Y''})$ be $n$-cotorsion pairs in $\A$ and $\C$ respectively. Set
    \begin{align*}
        \mathcal{X}=&\{B\in\B\mid j^{*}B\in\mathcal{X''}, i^{*}B\in \mathcal{X'}\},\\
       \mathcal{Y}=&\{B\in\B\mid j^{*}B\in\mathcal{Y''}, i^{!}B\in \mathcal{Y'}\}.
    \end{align*}
    Then
    \begin{itemize}
    \item[\rm (1)]
    $i_{*}\mathcal{X'}\subseteq\mathcal{X}$, $i_{*}\mathcal{Y'}\subseteq\mathcal{Y}$, $j_{!}\mathcal{X''}\subseteq\mathcal{X}$, $j_{*}\mathcal{Y''}\subseteq\mathcal{Y}$.
        \item[\rm (2)] $(\mathcal{X},\mathcal{Y})$ is an $n$-cotorsion pair in $\B$.
        \item[\rm (3)] $(\mathcal{X'},\mathcal{Y'})=(i^{*}\mathcal{X},i^{!}\mathcal{Y})$ and $(\mathcal{X''},\mathcal{Y''})=(j^{*}\mathcal{X},j^{*}\mathcal{Y})$.
    \end{itemize}
\end{corollary}

Applying Theorem \ref{main-thm1}
to $(n+1)$-cluster tilting subcategories, we get the following result.
\begin{corollary}\label{main-1-2}
       Let $(\A, \B, \C)$ be a recollement of extriangulated categories, and let $\mathcal{X'}$ and $\mathcal{X''}$ be $(n+1)$-cluster tilting subcategories in $\A$ and $\C$ respectively. Set
    \begin{align*}
        \mathcal{X}_1=&\{B\in\B\mid j^{*}B\in\mathcal{X''},\ i^{*}B\in \mathcal{X'}\},\\
       \mathcal{X}_2=&\{B\in\B\mid j^{*}B\in\mathcal{X''},\ i^{!}B\in \mathcal{X'}\}.
    \end{align*}
If $i^{*}$ and $i^!$ are exact, and $i^{*}j_{*}\mathcal{X''}\subseteq \mathcal{X'}$, $i^{!}j_{!}\mathcal{X''}\subseteq \mathcal{X'}$,
then
\begin{itemize}
\item[\rm (1)] $\mathcal{X}_1=\mathcal{X}_2$ is an $(n+1)$-cluster tilting subcategory in $\B$.
\item[\rm (2)] $\mathcal{X'}=i^{*}\mathcal{X}_1$ and $\mathcal{X''}=j^{*}\mathcal{X}_1$.
\end{itemize}

\end{corollary}

\begin{proof}
    (1)
    By Remark \ref{remark-nc}, we know that $(\mathcal{X'},\mathcal{X'})$ and $(\mathcal{X''},\mathcal{X''})$ are $(n+1)$-cotorsion pairs in $\A$ and $\C$ respectively.
Then
$(\mathcal{X}_1,\mathcal{X}_2)$ is an $n$-cotorsion pair in $\B$ by Theorem \ref{main-thm1}.

Let $X\in \mathcal{X}_1$. Since $i^*$ is exact by assumption, there exists the following $\E$-triangle
$$\xymatrix{j_{!}j^{*}X\ar[r]&X\ar[r]&i_{*}i^{*}X\ar@{-->}[r]&}$$
in $\B$ by Lemma \ref{lem-rec}.
Notice that $j^{*}X\in \mathcal{X''}$ and $i^{*}X\in \mathcal{X'}$.
Since
$i^{!}j_{!}\mathcal{X''}\subseteq\mathcal{X'}$ by assumption,
$i^{!}j_{!}j^{*}X\in i^{!}j_{!}\mathcal{X''}\subseteq \mathcal{X'}$.
Notice that $j^{*}i_{*}=0$, so $j^{*}i_{*}i^{*}X=0\in \mathcal{X''}$.
Since $i^{!}i_{*}\Rightarrow\Id_{\A}$ and $j^{*}j_{!}\Rightarrow \Id_{\C}$ by Lemma \ref{lem-rec}, $j^{*}j_{!}j^{*}X\cong j^{*}X\in \mathcal{X''}$ and $i^{!}i_{*}i^{*}X\cong i^{*}X\in \mathcal{X'}$.
So
$j_{!}j^{*}X\in \mathcal{X}_2$ and $i_{*}i^{*}X\in \mathcal{X}_2$.
Notice that $\mathcal{X}_2$ is closed under $\E$-extensions,
then $X\in \mathcal{X}_2$, and so $\mathcal{X}_1\subseteq \mathcal{X}_2$.
Similarly, we can prove that $\mathcal{X}_2\subseteq\mathcal{X}_1$.
Then $\mathcal{X}_1=\mathcal{X}_2$, and thus $\mathcal{X}_1=\mathcal{X}_2$ is an $(n+1)$-cluster tilting subcategory in $\B$ by Remark \ref{remark-nc}.

(2) This follows from Theorem \ref{main-thm1}(3).
\end{proof}

Specially, applying Corollary \ref{main-1-2} to triangulated categories, we get the result in \cite[Theorem 3.4]{LZZ24R}.
This gives a new proof of \cite[Theorem 3.4]{LZZ24R}.

\begin{corollary}{\rm (\cite[Theorem 3.4]{LZZ24R})}
\label{main-1-3}
     Let $(\A, \B, \C)$ be a recollement of triangulated categories, and let $\mathcal{X'}$ and $\mathcal{X''}$ be $(n+1)$-cluster tilting subcategories in $\A$ and $\C$ respectively. Set
    \begin{align*}
         \mathcal{X}_1=&\{B\in\B\mid j^{*}B\in\mathcal{X''},\ i^{*}B\in \mathcal{X'}\},\\
       \mathcal{X}_2=&\{B\in\B\mid j^{*}B\in\mathcal{X''},\ i^{!}B\in \mathcal{X'}\}.
    \end{align*}
If $i^{*}j_{*}\mathcal{X''}\subseteq \mathcal{X'}$ and $i^{!}j_{!}\mathcal{X''}\subseteq \mathcal{X'}$,
then
$\mathcal{X}_1=\mathcal{X}_2$ is an $(n+1)$-cluster tilting subcategory in $\B$.
\end{corollary}

\subsection{From $\B$ to $\A$ and $\C$}

%\begin{lemma}{\rm (\cite[Lemma 4.4]{MZ25D})}\label{lem-=}
%Let $(\A, \B, \C)$ be a recollement of extriangulated categories, and
%let $\X$ be a subcategory of $\mathcal{B}$. If $i_{*}i^{*}(\X) \subseteq \X$ and $i_{*}i^{!}(\X) \subseteq \X$, then $i^{*}(\X) =i^{!}(\X)$.
%\end{lemma}

\begin{lemma}\label{lem-=if}   %{\rm (cf. \cite[Lemma 4.5]{MZ25D})}
Let $(\A, \B, \C)$ be a recollement of extriangulated categories, and
let $(\X,\Y)$ be an $n$-cotorsion pair in $\B$. If $i^{*}$ and $i^!$ are exact, then we have
\begin{itemize}
\item[\rm (1)] $i_{*}i^{!}(\mathcal{Y})\subseteq \mathcal{Y}$ if and only if $i_{*}i^{*}(\mathcal{X})\subseteq \mathcal{X}$.
\item[\rm (2)] $j_{*}j^{*}(\mathcal{Y})\subseteq \mathcal{Y}$ if and only if $j_{!}j^{*}(\mathcal{X})\subseteq \mathcal{X}$.
\end{itemize}
\end{lemma}
\begin{proof}
For any $X \in \mathcal{X}$, $Y \in \mathcal{Y}$ and $1\leq k\leq n$, we have
$$
\begin{aligned}
& \E^{i}_\B(X,i_{*}i^{!}(Y))\cong \E^{i}_\A(i^{*}(X),i^{!}(Y))\cong  \E^{i}_\B(i_{*}i^{*}(X),Y)\\
& \E^{i}_\B(X,j_{*}j^{*}(Y))\cong \E^{i}_\C(j^{*}(X),j^{*}(Y))\cong  \E^{i}_\B(j_{!}j^{*}(X),Y)
\end{aligned}
$$
by Proposition \ref{E-xt-1}.
It follows that the assertions (1) and (2) are true.
\end{proof}

In contrast to Theorem~\ref{main-thm1}, which provides a method for constructing an $n$-cotorsion pair in the middle term $\mathcal{B}$ from those in the outer terms $\mathcal{A}$ and $\mathcal{C}$, the next result addresses the converse problem. Specifically, we investigate how an $n$-cotorsion pair in $\mathcal{B}$ can induce $n$-cotorsion pairs in $\mathcal{A}$ and $\mathcal{C}$ under suitable exactness and closure conditions. This converse construction is summarized in the following theorem.

\begin{theorem}\label{main-thm2}
Let $(\A, \B, \C)$ be a recollement of extriangulated categories, and let $(\mathcal{X},\mathcal{Y})$ be an $n$-cotorsion pair in $\B$. Assume that $i^{*}$ and $i^!$ are exact, then we have the following statements.
\begin{itemize}
\item[\rm (1)]  If $i_{*}i^{*}\mathcal{X}\subseteq\mathcal{X}$, $i_{*}i^{*}\mathcal{Y}\subseteq\mathcal{Y}$, then
$(i^{*}\mathcal{X},i^{!}\mathcal{Y})$ is an $n$-cotorsion pair in $\A$.
\item[\rm (2)] If $j_{*}j^{*}\mathcal{Y}\subseteq\mathcal{Y}$, then
$(j^{*}\mathcal{X},j^{*}\mathcal{Y})$ is an $n$-cotorsion pair in $\C$.
\item[\rm (3)] If $i_{*}i^{*}\mathcal{X}\subseteq\mathcal{X}$, $i_{*}i^{*}\mathcal{Y}\subseteq\mathcal{Y}$, and $j_{*}j^{*}\mathcal{Y}\subseteq\mathcal{Y}$, then
\begin{align*}
\mathcal{X}=&\{B\in \B\mid j^{*}B\in j^{*}\mathcal{X},\  i^{*}B\in i^{*}\mathcal{X}\},\\
\mathcal{Y}=&\{B\in \B\mid j^{*}B\in j^{*}\mathcal{Y},\  i^{!}B\in i^{!}\mathcal{Y}\}.
\end{align*}
\end{itemize}
\end{theorem}

\begin{proof}
Since $(\X,\Y)$ is a $n$-cotorsion pair in $\B$ by assumption,
$\mathcal{X}$ is contravariantly finite and $\mathcal{Y}$ is covariantly finite in $\B$.

(1)
Since $(i^{*},i_{*})$ and $(i_{*},i^{!})$ are adjoint pairs,
$i^{*}\mathcal{X}$ is contravariantly finite and $i^{!}\mathcal{Y}$ is covariantly finite in $\A$
by Lemma \ref{lem-con-cov}.
Since $i_{*}i^{*}\mathcal{X}\subseteq \mathcal{X}$ by assumption, $i_{*}i^{!}\mathcal{Y}\subseteq \mathcal{Y}$ by Lemma \ref{lem-=if}.
%Notice that $i_{*}i^{*}\Y\subseteq\Y$ by assumption, then $i^{*}\Y=i^{!}\Y$ by Lemma \ref{lem-=}.

Now we claim that $i^{*}\mathcal{X}=\bigcap\limits_{i=1}^{n}{^{\perp_i}(i^{!}\mathcal{Y})}.$
Let $X'\in i^{*}\X$, that is, there exists some $X\in \X$ such that $X'= i^{*}X$. For any $Y'\in i^!\mathcal{Y}$ and $1\leq k\leq n$, there exists some $Y\in \Y$ suc that $Y'=i^{!}Y$, then we have
$$\E^{k}_{\A}(X',Y')\cong\E^{k}_{\A}(i^{*}X,i^{!}Y)\cong \E^{k}_{\B}(X,i_{*}i^{!}Y)=0$$
by Proposition \ref{E-xt-1} and the fact that $i_{*}i^{!}\Y\subseteq\Y$.
Then $X'\in \bigcap\limits_{i=k}^{n}{^{\perp_k}(i^{!}\mathcal{Y})}$, and so $i^{*}\mathcal{X}\subseteq\bigcap\limits_{k=1}^{n}{^{\perp_k}(i^{!}\mathcal{Y})}.$
Conversely, let $X'\in \bigcap\limits_{k=1}^{n}{^{\perp_k}(i^{!}\mathcal{Y})}$.
Then
$$\E^{k}_{\B}(i_{*}X',Y)\cong \E^{k}_{\A}(X',i^{!}Y)=0$$
for any $Y\in \Y$ and $1\leq k\leq n$.
It follows that $i_{*}X'\in \bigcap\limits_{k=1}^{n}{^{\perp_k}\mathcal{Y}}=\mathcal{X}$.
Then $X'\cong i^{*}i_{*}X'\in i^{*}\mathcal{X}$ by the fact that $i^{*}i_{*} \Rightarrow\Id_{\A}$.
So $\bigcap\limits_{k=1}^{n}{^{\perp_k}(i^!\mathcal{Y})}\subseteq i^{*}\X$.
Thus $i^{*}\X=\bigcap\limits_{k=1}^{n}{^{\perp_k}(i^!\mathcal{Y})}$.

Similarly, we can prove that
$i^{!}\mathcal{Y}=\bigcap\limits_{k=1}^{n}{(i^{*}\mathcal{X})^{\perp_k}}.$
Then $(i^{*}\mathcal{X},i^{!}\mathcal{Y})$ is an $n$-cotorsion pair in $\A$.

(2) Since $(j_{!},j^{*})$ and $(j^{*},j_{*})$ are adjoint pairs,
 $j^{*}\mathcal{X}$ is contravariantly finite and $j^{*}\mathcal{Y}$ is covariantly finite in $\C$
by Lemma \ref{lem-con-cov}.

Now we claim that
$j^{*}\mathcal{X}=\bigcap\limits_{k=1}^{n}{^{\perp_k}(j^{*}\mathcal{Y})}.$

Let $X''\in j^{*}\mathcal{X}$, that is, there exists some $X\in \X$ such that $X''=j^{*}X$. For any $Y''\in j^*\mathcal{Y}$ and $1\leq k\leq n$, there exists some $Y\in \Y$ such that $Y''=j^*Y$, then we have
$$\E^{k}_{\C}(X'',Y'')\cong \E^{k}_{\C}(j^{*}X,j^{*}Y)\cong \E^{k}_{\B}(X,j_{*}j^{*}Y)=0$$
by Proposition \ref{E-xt-1} and the assumption that $j_{*}j^{*}\mathcal{Y}\subseteq \mathcal{Y}$.
Then $X''\in\bigcap\limits_{k=1}^{n}{^{\perp_{k}}(j^{*}\mathcal{Y})}$, it follows that
$j^{*}\mathcal{X}\subseteq\bigcap\limits_{k=1}^{n}{^{\perp_k}(j^{*}\mathcal{Y})}.$

On the other hand,
let $X''\in \bigcap\limits_{k=1}^{n}{^{\perp_k}(j^{*}\mathcal{Y})}$.
Then for any $Y\in \Y$ and $1\leq k\leq n$,
$$\E^{k}_{\B}(j_{!}X'',Y)\cong \E^{k}_{\C}(X'',j^{*}Y)=0$$
by Proposition \ref{E-xt-1}.
It follows that $j_{!}X''\in \bigcap\limits_{k=1}^{n}{^{\perp_k}\mathcal{Y}}=\mathcal{X}$.
Since $j^{*}j_{!}\Rightarrow\Id_{\C}$ by Lemma \ref{lem-rec},
$X''\cong j^{*}j_{!}X''\in j^{*}\mathcal{X}$.
Then $\bigcap\limits_{k=1}^{n}{^{\perp_k}(j^{*}\mathcal{Y})}\subseteq j^{*}\mathcal{X}.$

Similarly, we can prove that
$j^{*}\mathcal{Y}=\bigcap\limits_{i=1}^{n}({j^{*}\mathcal{X})^{\perp_i}}.$
Thus $(j^{*}\mathcal{X},j^{*}\mathcal{Y})$ is an $n$-cotorsion pair in $\C$.

(3)
It is clear that
\begin{align*}
\mathcal{X}\subseteq&\{B\in \B\mid j^{*}B\in j^{*}\mathcal{X},\  i^{*}B\in i^{*}\mathcal{X}\},\\
\mathcal{Y}\subseteq&\{B\in \B\mid j^{*}B\in j^{*}\mathcal{Y},\  i^{!}B\in i^{!}\mathcal{Y}\}.
\end{align*}
Conversely, suppose $X\in \{B\in \B\mid j^{*}B\in j^{*}\X,\ i^{*}B\in i^{*}\X\}$.
Since $i^{*}$ is exact, there exists an $\mathbb{E}$-triangle
$$
  \xymatrix{j_! j^\ast X\ar[r]&X\ar[r]&i_\ast i^\ast X\ar@{-->}[r]&}
$$
in $\mathcal{B}$ by Lemma \ref{lem-rec}.
For any $Y \in \mathcal{Y}$ and $1\leq k\leq n$, applying the functor to the above $\E$-triangle yields the following exact sequence
$$
\xymatrix@C=15pt{\cdots\ar[r]&
\E^{k}_{\B} (i_{*}i^{*}(X), Y) \ar[r]&\E^{k}_{\B} (X, Y) \ar[r]&\E^{k}_{\B}  (j_{!}j^{*}(X), Y)\ar[r]&\cdots}.$$
By (1) and (2), we know that  $(i^{*}\mathcal{X}, i^{!}\mathcal{Y})$ and $(j^{*}\X,j^{*}\Y)$ are $n$-cotorsion pairs in $\A$ and $\C$ respectively, so $\E^{k}_{\B}(i_{*}i^{*}X, Y)\cong \E^{k}_{\A}(i^{*}X, i^{!}Y)=0$ and $\E^{k}_{\B}(j_{!}j^{*}X, Y)\cong \E^{k}_{\C}(j^{*}X, j^{*}Y)=0$ by Proposition \ref{E-xt-1}.
It follows that $\E^{i}_{\B}(X,Y)=0$, and so $X\in\bigcap\limits_{k=1}^{n}{^{\perp_k}\mathcal{Y}}=\mathcal{X}$.
Thus
 $$\{B\in \B\mid j^{*}B\in j^{*}\X,\ i^{*}B\in i^{*}\X\}\subseteq\X.$$
 Similarly, we can prove
 $$\{B\in \B\mid j^{*}B\in j^{*}\Y,\ i^{!}B\in i^{!}\Y\}\subseteq\Y.$$
\end{proof}

Applying Theorem \ref{main-thm2} to triangulated categories, we get the following result, which is a generalization of the result in \cite[Theorem 3.3]{C13C}.

\begin{corollary}\label{main-2-1}
Let $(\A, \B, \C)$ be a recollement of triangulated categories, and let $(\mathcal{X},\mathcal{Y})$ be an $n$-cotorsion pair in $\B$. Then we have the following statements.
\begin{itemize}
\item[\rm (1)]  If $i_{*}i^{*}\mathcal{X}\subseteq\mathcal{X}$, $i_{*}i^{*}\mathcal{Y}\subseteq\mathcal{Y}$, then
$(i^{*}\mathcal{X},i^{!}\mathcal{Y})$ is an $n$-cotorsion pair in $\A$.
\item[\rm (2)] If $j_{*}j^{*}\mathcal{Y}\subseteq\mathcal{Y}$, then
$(j^{*}\mathcal{X},j^{*}\mathcal{Y})$ is an $n$-cotorsion pair in $\C$.
\item[\rm (3)] If $i_{*}i^{*}\mathcal{X}\subseteq\mathcal{X}$, $i_{*}i^{*}\mathcal{Y}\subseteq\mathcal{Y}$, and $j_{*}j^{*}\mathcal{Y}\subseteq\mathcal{Y}$, then
\begin{align*}
\mathcal{X}=&\{B\in \B\mid j^{*}B\in j^{*}\mathcal{X},\  i^{*}B\in i^{*}\mathcal{X}\},\\
\mathcal{Y}=&\{B\in \B\mid j^{*}B\in j^{*}\mathcal{Y},\  i^{!}B\in i^{!}\mathcal{Y}\}.
\end{align*}
\end{itemize}
\end{corollary}

When $n=1$ in Theorem \ref{main-thm2}, we get the following result.

\begin{corollary}{\rm (\cite[Theorem 4.6]{MZ25D} and \cite[Theorem 4.4]{WWZ20R})}\label{main-2-2}
Let $(\A, \B, \C)$ be a recollement of extriangulated categories, and let $(\mathcal{X},\mathcal{Y})$ be a otorsion pair in $\B$. Assume that $i^{*}$ and $i^!$ are exact, then we have the following statements.
\begin{itemize}
\item[\rm (1)]  If $i_{*}i^{*}\mathcal{X}\subseteq\mathcal{X}$, $i_{*}i^{*}\mathcal{Y}\subseteq\mathcal{Y}$, then
$(i^{*}\mathcal{X},i^{!}\mathcal{Y})$ is a cotorsion pair in $\A$.
\item[\rm (2)] If $j_{*}j^{*}\mathcal{Y}\subseteq\mathcal{Y}$, then
$(j^{*}\mathcal{X},j^{*}\mathcal{Y})$ is a cotorsion pair in $\C$.
\item[\rm (3)] If $i_{*}i^{*}\mathcal{X}\subseteq\mathcal{X}$, $i_{*}i^{*}\mathcal{Y}\subseteq\mathcal{Y}$, and $j_{*}j^{*}\mathcal{Y}\subseteq\mathcal{Y}$, then
\begin{align*}
\mathcal{X}=&\{B\in \B\mid j^{*}B\in j^{*}\mathcal{X},\  i^{*}B\in i^{*}\mathcal{X}\},\\
\mathcal{Y}=&\{B\in \B\mid j^{*}B\in j^{*}\mathcal{Y},\  i^{!}B\in i^{!}\mathcal{Y}\}.
\end{align*}
\end{itemize}
\end{corollary}

Applying  Corollary \ref{main-2-2} to triangulated categories, we get the following.

\begin{corollary}\label{main-2-3}
{\rm(\cite[Theorem 3.3]{C13C})}
Let $(\A, \B, \C)$ be a recollement of triangulated categories, and let $(\mathcal{X},\mathcal{Y})$ be a otorsion pair in $\B$. Then we have the following statements.
\begin{itemize}
\item[\rm (1)]  If $i_{*}i^{*}\mathcal{X}\subseteq\mathcal{X}$, $i_{*}i^{*}\mathcal{Y}\subseteq\mathcal{Y}$, then
$(i^{*}\mathcal{X},i^{!}\mathcal{Y})$ is a cotorsion pair in $\A$.
\item[\rm (2)] If $j_{*}j^{*}\mathcal{Y}\subseteq\mathcal{Y}$, then
$(j^{*}\mathcal{X},j^{*}\mathcal{Y})$ is a cotorsion pair in $\C$.
\item[\rm (3)] If $i_{*}i^{*}\mathcal{X}\subseteq\mathcal{X}$, $i_{*}i^{*}\mathcal{Y}\subseteq\mathcal{Y}$, and $j_{*}j^{*}\mathcal{Y}\subseteq\mathcal{Y}$, then
\begin{align*}
\mathcal{X}=&\{B\in \B\mid j^{*}B\in j^{*}\mathcal{X},\  i^{*}B\in i^{*}\mathcal{X}\},\\
\mathcal{Y}=&\{B\in \B\mid j^{*}B\in j^{*}\mathcal{Y},\  i^{!}B\in i^{!}\mathcal{Y}\}.
\end{align*}
\end{itemize}
\end{corollary}

Applying Theorem \ref{main-thm2} to $(n+1)$-cluster tilting subcategories, we get the following result.

\begin{corollary}\label{main-2-4}
         Let $(\A, \B, \C)$ be a recollement of extriangulated categories, and
let $\mathcal{X}$ be an $(n+1)$-cluster tilting subcategory in $\B$. Assume that $i^{*}$ and $i^!$ are exact, then we have the following statements.
\begin{itemize}
\item[\rm (1)] If $i_{*}i^{*}\mathcal{X} \subseteq \mathcal{X}$, then $i^{*}\mathcal{X}$ is an $(n+1)$-cluster tilting subcategory in $\A$.

\item[\rm (2)] If $j_{*}j^{*}\mathcal{X} \subseteq \mathcal{X}$, then $j^{*}\mathcal{X}$ is an $(n+1)$-cluster tilting subcategory in $\C$.

\item[\rm(3)] If $i_{*}i^{*}\mathcal{X}\subseteq \mathcal{X}$ and $j_{*}j^{*}\mathcal{X}\subseteq \mathcal{X}$, then we have
\begin{align*}
\mathcal{X}=&\{B\in \B\mid j^{*}B\in j^{*}\mathcal{X}, i^{*}B\in i^{*}\mathcal{X}\}.
\end{align*}
    \end{itemize}
\end{corollary}

Specially, applying Corollary \ref{main-2-4} to triangulated categories, we get the result in \cite[Theorem 4.4]{LZZ24R}.

\begin{corollary}\label{main-2-5}
    {\rm (\cite[Theorem 4.4]{LZZ24R})}
      Let $(\A, \B, \C)$ be a recollement of triangulated categories, and
let $\mathcal{X}$ be an $(n+1)$-cluster tilting subcategory in $\B$. Then we have the following statements.
\begin{itemize}
\item[\rm (1)] If $i_{*}i^{*}\mathcal{X} \subseteq \mathcal{X}$, then $i^{*}\mathcal{X}$ is an $(n+1)$-cluster tilting subcategory in $\A$.

\item[\rm (2)] If $j_{*}j^{*}\mathcal{X}\subseteq \mathcal{X}$, then $j^{*}\mathcal{X}$ is an $(n+1)$-cluster tilting subcategory in $\C$.
    \end{itemize}
\end{corollary}

\vspace{5mm}

\hspace{-6mm}\textbf{Data Availability}\hspace{2mm} Data sharing not applicable to this article as no datasets were generated or analysed during
the current study.
\vspace{2mm}

\hspace{-6mm}\textbf{Conflict of Interests}\hspace{2mm} The authors declare that they have no conflicts of interest to this work.

{\bf Xin Ma}\\
College of Science, Henan University of Engineering, 451191 Zhengzhou, Henan, P. R. China\\
E-mail: maxin@haue.edu.cn
\\[0.3cm]
\textbf{Panyue Zhou}\\
School of Mathematics and Statistics, Changsha University of Science and Technology, 410114 Changsha, Hunan,  P. R. China\\
E-mail: panyuezhou@163.com

\end{document}